\documentclass{article}
\usepackage{graphicx} 
\usepackage{amsfonts}
\usepackage{authblk}
\usepackage{amsmath}
\usepackage{amssymb}
\usepackage{amsthm}
\usepackage[english]{babel}
\usepackage{mathtools}
\usepackage{mathpazo}
\usepackage{tikz-cd}
\usepackage[all]{xy}
\title{Twisted cohomology on algebraic and analytic varieties}
\author[1]{Md. Shariful Islam}
\author[2]{Abdur-Rahman Mishkaat}
\affil[1,2]{Department of Mathematics, University of Dhaka}
\date{May, 2026}

\begin{document}
\theoremstyle{definition}
\newtheorem{definition}{Definition}[section]
\newtheorem{theorem}{Theorem}
\theoremstyle{remark}
\newtheorem*{remark}{Remark}
\maketitle

\section{Introduction}
On a smooth manifold $M$, we know that by de Rham's theorem that the singular cohomology of $M$ can be computed from the cohomology of differential forms on $M$. More precisely, from the cochain complex of differential forms $(\Omega^{\bullet}(M),d:\Omega^{k}(M)\longrightarrow \Omega^{k+1}(M))$, we compute the cohomomology following the usual procedure. In this article, we are interested in a rather deformed or \textit{twisted} version of the de Rham complex and it's cohomology. For any closed 1-form $\omega\in\Omega^{1}(M)$, we may consider the twisted exterior derivative 
\[
d_{\omega}:= d + \omega\wedge
\]
where the closedness of $\omega$ ensures that $d_{\omega}^{2} =0$. Thus, we have 
\begin{definition}
The cohomology of the twisted de Rham complex $(\Omega^{\bullet}(M),d_{\omega}^{k}: \Omega^{k}(M)\longrightarrow\Omega^{k+1}(M))$ is called the twisted cohomology of $M$
\end{definition}
\begin{equation}
    \operatorname{H}_{\omega}^{k}(M) := \frac{\operatorname{Ker}(d_{\omega}^{k})}{\operatorname{Im}(d_{\omega}^{k-1})}
\end{equation}Such cohomologies have been studied before in several different contexts \cite{monnier2005cohomology},\cite{ida2012basic},\cite{ornea2009morse}. They were typically called \textit{Morse-Novikov} cohomology or \textit{Lichnerowicz} cohomology. On smooth manifolds that are foliated such cohomologies have some properties like homotopy invariance and Poincare duality leafwise (leaves of the foliation) that is not true for ordinary de Rham cohomology groups. Such a case is studied in \cite{islam2019leafwise,islam2024morse}. However, for our narrowed interest towards it's connection with analytic and algebraic structural deformations, we are tempted to call it \textit{twisted} cohomology.\\
Now let us proceed to specify exactly which type manifolds we are interested in here in this article. As can be guessed from the title, we restrict our emphasis to varieties that are both (complex) analytic and algebraic in nature. We shall define the algebraic version of the twisted cohomologies through deforming the algebraic de Rham complex with a Kahler differential. One of our motivations for doing so was the work done in \cite{meng2018explicit,meng2020morse,meng2023morse}, that was done on complex manifolds. The author of those papers deformed (or we can say "\textit{twisted}", as in our terminology) the Dolbeault complex with a $\bar{\partial}$-closed $(0,1)$-form $\eta$ as $\bar{\partial}\longrightarrow \bar{\partial}_{\eta}:= \bar{\partial} + \eta\wedge$, where the $\bar{\partial}$-closedness of $\eta$ again is used to show the nilpotency of $\bar{\partial}_{\eta}$. We recall the definition from \cite{meng2020morse} here 
\begin{definition}
    On a complex manifold $X$, the cohomology of the (twisted) complex $(\Omega^{\bullet,\bullet}(X),\bar{\partial}_{\eta})$ is called the Dolbeault-Morse-Novikov cohomology. 
\end{definition}
Yet, in our terminology, we shall call it the twisted Dolbeault cohomology. Parallel to our language choice and for later convenience we give the following nomenclature 
\begin{definition}
    Any closed differential form $\psi$ that is used to deform the differential operator or derivative of a (co)chain complex $(\Theta^{\bullet},\delta)$ to $(\Theta^{\bullet},\delta_{\psi})$ is called a \textit{twisting} parameter.   
\end{definition}
Thus the $\eta\in\Omega^{0,1}(X)$, $\omega\in\Omega^{1}(M)$ that we mentioned above are the twisting parameters of the corresponding chain complex.\\
We want to define such twisted cohomologies on an algebraic variety and compare it with the twisted cohomologies studied for the analytic cases in the literature. Such cohomologies on were already studied in the context complex algebraic geometry in \cite{kita1994vanishing}. However, there are some differences between the twisting parameters in our study and those in \cite{kita1994vanishing}, where logarithmic aspects of twisted cohomology is much more widely discussed and studied.  As we shall see, the existence of an algebraic twisting parameter on an algebraic variety is tied to stringent constraints that must be imposed on the variety. Moreover, in the complex analytic case, twisted cohomologies are seen to be tied to holomorphic deformations of the complex structure on the complex manifold. This is connected to the fact that the Dolbeault cohomology has some \textit{sensitivity} towards the complex structures. But the overall de Rham cohomology obtained from the direct sum of the Dolbeault pieces has no such \textit{sensitivity} and is entirely topological.\\
\textit{ The organization of the paper}: In section we first reviewed GAGA results related to cohomologies and Dolbeault isomorphism. Then we highlighted two different approaches of defining twisting parameters. We included some example computations that are at the \textit{homework} level difficulty. In section three, we discussed the weight sheaves and a theorem of Deligne. And then in the fourth section we describe algebraic twists and have included some easy computations. 
\section{GAGA review} Before defining the twisted algebraic cohomology, we want to recall the constraints required for making comparisons of it with the ones defined for analytic varieties.  First of all, we must consider algebraic manifolds that are algebraic varieties with an underlying manifold structure equipped on the underlying space of the variety. Direct comparisons are available in the literature in the form of the so called GAGA theorems \cite{serre2021algebraic}. Henceforth we shall and algebraic variety by $X$ and the underlying smooth analytic manifold $X^h$.\footnote{We are using analytic varieties and analytic manifolds synonymously, although they are not entirely equivalent} GAGA theorem requires the variety $X$ to be projective. So on the analytic side, the manifold $X^{h}$ can be a smooth submanifold of a complex projective space. The GAGA analytification functor is essentially tied to the following map 
\begin{equation}
    \lambda: X^{h}\longrightarrow X
\end{equation}
which, when viewed from the topological lens, is a topology weakening map since the Zariski topology on the variety $X$ is much weaker than the smooth topology on $X^h$. So the twisted Dolbeault cohomology that we discussed above is a cohomology on $X^h$ 
\[
\operatorname{H}_{\eta}^{p,q}(X^{h})
\] 
Here's the first constraint we get from such considerations. While on the complex manifold $X^h$, we can take a $(0,1)$-form that is closed under the action of $\bar{\partial}$, we first of all have no algebraic differential forms (i.e. K\"{a}hler differentials) on a projective variety obtained from the algebraic exterior derivative of the global regular functions since all the global regular (algebraic) functions on a projective variety must be constant i.e. 
\[\operatorname{H}^{0}(X,\mathcal{O}_{X}) \simeq \mathbb{C}\]
on the complex side, the manifold $X^{h}$ being compact, Liouville's theorem implies that the global constant functions must be constant too. So, on the algebraic variety $X$, how do we get the twisting parameter? The natural solution to this stumbling block turns out to be equipping the structure of an abelian group on the projective variety $X$, that is, $X$ has to be \textbf{an abelian variety}.\\
Informally speaking, from the group structure of the variety (we are working over $\mathbb{C}$), we may define a (Zariski) tangent space at the point that is also the identity element from the group theoretic perspective. The variety being a complex Lie group now, the tangent space at the identity is the Lie algebra. Then the dual tangent space i.e. the cotangent space has cotangent vectors that are the differential 1-forms on it.\\
Yet another way of justifying this at the complex dimension 1, for example, is to look at the Riemann surfaces of genus $g\geq 1$. If the genus is 1, we have an elliptic curve over the complex numbers; it is also a Calabi Yau one-fold. We shall study this case later on. From the algebraic viewpoint, this is a projective zero set of a homogeneous degree 3 polynomial. For higher genera, we call them hyper-elliptic curves. Under some special conditions, the hyper elliptic curves are algebraic too. Then the Riemann-Roch theorem implies that on a genus $g$ Riemann surface, there are precisely $g$ holomorphic 1-forms, that are closed.\\
Therefore, the twisting form to exist on the variety $X$, it must be an abelian variety.\\

\begin{remark} In fact, on any projective smooth variety we can define sheaf of K\'{a}hler differentials from the sheafification of the presheaf of cotangent modules on the underlying graded algebra. But it has no global section. The twisting parameter must be globally defined. Therefore, we shall restrict our attention to abelian varieties.  
\end{remark}

\subsection{Dolbeault isomorphism} Algebraic de Rham cohomology was defined in \cite{grothendieck1966rham,grothendieck1977cohomologie,grothendieck2005cohomologie,hartshorne1975rham} as the hyper cohomology of the algebraic de Rham complex 
\[
\operatorname{H}^{p}(X/k) := \operatorname{\textbf{H}}^{p}(X,\Omega^{\bullet})
\]
where the hypercohomology is reduced to cohomology of the chain complex for affine varieties. We recall that we cannot directly connect Dolbeault cohomology to the cohomologies on algebraic varieties. Because, on the complex manifold $X^h$ we have the decomposition of the differential forms under the action of complexification 
\begin{equation}
    \Omega^{k}(X^h) = \bigoplus_{p+q=k}\Omega^{p,q}(X^h)
\end{equation} 
which has no algebraic analog for algebraic differential forms on algebraic varieties. Rather, the identification is done via applying Dolbeault isomorphism combined with GAGA. Dolbeault theorem gives the sheaf cohomological isomorphism 
\begin{equation}
    \operatorname{H}_{\overline{\partial}}^{p,q}(X^h) \simeq \operatorname{H}^{q}(X^h,\Omega^{p,0}_{X^h})
\end{equation}
where on the right hand side, $\Omega_{X^h}^{p,q}$ is understood to be the sheaf of holomorphic $p$-forms on $X^h$. Then GAGA theorem gives the following identification 
\begin{equation}
    \operatorname{H}^{q}(X^h,\Omega^{p,0}_{X^h}) \simeq \operatorname{H}^{q}(X,\Omega^{p}_{X})
\end{equation}
where now on the right hand side, $\Omega^{p}_{X}$ is the sheaf of algebraic $p$-forms on the variety $X$.\\
However, there are some subtleties in applying GAGA for the twisted Dolbeault cohomology. Returning to Meng's definition \cite{meng2018explicit,meng2020morse,meng2023morse} of the twisted Dolbeault cohomology, the twisting parameter $\eta\in\Omega^{0,1}_{X^h}$ we may take into account the complex conjugation of any $(0,1)$-form is a $(1,0)$-form and vice versa.\\

On the variety $X$, let $\psi$ be a twisting parameter on $X$. Then GAGA theorem implies that on the analytic space $X^h$, there is a corresponding holomorphic 1-form $\psi^h\in \Omega_{X^h}^{1,0}$. Schematically  
\[
 \psi \overset{\text{GAGA}}{\longleftrightarrow} \psi^{h} 
\]
Now we describe below two ways of identifying the algebraic twisting parameter with the analytic twisting parameter:
\begin{enumerate}
    \item The complex conjugate of the holomorphic 1-form $\overline{\psi}^{h} = : \eta\in \Omega^{0,1}_{X^h}$ and define the twisted Dolbeault $\bar{\partial}_{\eta}$. The closedness of $\eta$ is implied by the algebraic closedness of the algebraic 1-form $\psi$.\ Hence we give 
    \begin{definition}
        The analytic twisting parameter $\eta\in\Omega_{X^h}^{0,1}$ obtained from the complex conjugation of the holomorphic 1-form $\psi^h$ that is again from an algebraic 1-form $\psi$ on the variety X, is called a \textit{conjugated} twisting parameter. 
    \end{definition}
    \item We contract the holomorphic vector valued Beltrami differential $\mu\in\Omega^{0,1}_{X^h}(T^{1,0}_{X^h})$ on the complex manifold $X^h$, with the analytic twisting parameter $\psi^{h}$. In local coordinates, the holomorphic vector valued (0,1) form may be written as 
    \begin{equation}
        \mu(z) = \mu_{\overline{j}}^{i}(z)d\overline{z}^{\bar{j}}\otimes \frac{\partial}{\partial z^{i}} 
    \end{equation}
    and the holomorphic twisting parameter is 
    \begin{equation}
        \psi^{h} = \psi_{k}^{(h)}dz^{k} \in \Omega^{1,0}_{X^h}
    \end{equation}
    Then the contraction can be performed as follows:
    \begin{align*}
        \psi^{h}(\mu) &= \langle \psi^{(h)}_{k}dz^{k}, \mu_{\overline{j}}^{i}(z)d\overline{z}^{\bar{j}}\otimes \frac{\partial}{\partial z^{i}}  \rangle\\
        &=  \psi^{(h)}_{k}\mu_{\overline{j}}^{i}(z) \langle dz^{k}, \frac{\partial}{\partial z^{i}}  \rangle d\overline{z}^{\bar{j}}\\
        &=  \psi^{(h)}_{k}\mu_{\overline{j}}^{i}(z) \delta^{k}_{i} d\overline{z}^{\bar{j}}\\
        &=  \psi^{(h)}_{i}\mu_{\overline{j}}^{i}(z) d\overline{z}^{\bar{j}} \in \Omega^{0,1}_{X^h}\\    
    \end{align*}
    We restrict our attention to the $\mu$'s that satisfy  the Maurer-Cartan equation 
    \begin{equation}
        \overline{\partial}\mu + \frac{1}{2}[\mu,\mu] =0
    \end{equation} However, to ensure the closedness of (0,1)-form obtained from the contraction as shown above, we see that the (Frolicher-Nijenhuis) bracket term $[\mu,\mu]$ term is an obstruction. Denoting the (0,1)-form as $\eta$ again, taking the anti-holomorphic derivative 
    \begin{align*}
        \overline{\partial}\eta_{\bar{j}} &= \overline{\partial}(\psi^{(h)}_{i}\mu^{i}_{\bar{j}})d\bar{z}^{\bar{j}}\\
        &= \psi^{(h)}_{i}(\bar{\partial}\mu_{\bar{j}}^{i})d\bar{z}^{\bar{j}}\\  &= \psi^{(h)}_{i}(-\frac{1}{2}[\mu,\mu]_{\bar{j}}^{i})d\bar{z}^{\bar{j}}\\ 
    \end{align*}
    Thus, the closedness of the analytic twisting parameter $\eta$ is equivalent to the condition 
    \begin{equation}
        [\mu,\mu] = 0
    \end{equation} 
    This motivates the following 
    \begin{definition}
        An analytic twisting parameter $\eta\in\Omega_{X^h}^{0,1}$ obtained from contracting a holomorphic vector valued (0,1)-form $\mu$ such that $[\mu,\mu]=0$, with the holomoprhic 1-form $\psi^{h}\in\Omega_{X^{h}}^{1,0}$ obtained again from an algebraic twisting parameter $\psi$ on an algebraic abelian variety $X$ is called a \textit{contracted} twisting parameter. 
    \end{definition}
\end{enumerate} 
\subsubsection{Conjugated twisting parameter}
As a remark on the first way of identifying the algebraic twisting parameter with the analytic one, we mention some features. First of all, the conjugated twisting parameter is obtained \textit{indirectly}. We will later see a different way of viewing the twisting cohomologies via line bundles. Namely, the twisted cohomology is obtained from a flat line bundle whose connection 1-form is precisely the twisting parameter, where the flatness of the (connection of the) line bundle is equivalent to the nilpotency of the twisted derivative. This way of studying twisted cohomologies has been studied by several authors cf \cite{ornea2009morse,ornea2010locally}, for example. The line bundle in this case is typically called \textit{weight bundle}. We may recall from the literature that for a pair $(X,X^h)$ with the analytification map $\lambda: X^h \longrightarrow X$ the associated Picard varieties that parametrize degree zero (i.e. first Chern vanishing) algebraic line bundles (equivalently, locally free sheaf of $\mathcal{O}_{X}$-modules of rank 1) over $X$ and holomorphic line bundles of degree zero over $X^h$ i.e. $\operatorname{Pic}^{0}(X)$ and $\operatorname{Pic}^{0}_h(X^h)$ are identified. Yet another hint of the existence of algebraic twisted cohomology closely related to the twisted Dolbeault cohomology. We will come back to this perspective later.\\
For two cohomologous (in the sense of Dolbeault cohomology on $X^h$) we have the following 
\begin{theorem} For two cohomologous analytic twisting parameter $\psi^{h}_{1},\psi^{h}_{2}\in\Omega^{0,1}_{X^h}$, the two twisted Dolbeault cohomologies are isomorphic 
\begin{equation}
    \operatorname{H}^{p,q}_{\overline{\partial}_{\psi^{h}_{1}}}(X^h) \simeq \operatorname{H}^{p,q}_{\overline{\partial}_{\psi^{h}_{2}}}(X^h)
\end{equation}
\end{theorem}
\textit{Proof 1}: For two $\overline{\partial}$-closed $(0,1)$-forms $\psi^{h}_{1},\psi^{h}_{2}$ that differ by an $\overline{\partial}$-exact form 
\begin{equation}
\psi^{h}_{1} = \psi^{h}_{2} + \frac{\overline{\partial}f}{f}
 \end{equation}
where $f$ is some locally defined smooth complex valued function on some open set $U$ of the manifold $X^{h}$, where the locality of the domain of the function $f$ follows from the Grothendieck-Poincar\'e lemma. Because of the compactness of $X^{h}$, globally non vanishing holomorphic functions on $X^{h}$ are constant functions. So, we will be comparing the groups $H^{p,q}_{\psi^{h}_{1}}(U)$, $H^{p,q}_{\psi^{h}_{2}}(U)$. Clearly for such $\psi^{h}$'s the twisted operators are
\begin{equation}
    \overline{\partial}_{\psi^{h}_{1}}:= \overline{\partial} + \psi^{h}_{1}\wedge
\end{equation}
\begin{equation}
    \overline{\partial}_{\psi^{h}_{2}}:= \overline{\partial} + \psi^{h}_{2}\wedge
\end{equation}
which are related as 
\begin{equation}
    \overline{\partial}_{\psi^{h}_{1}} = \overline{\partial}_{\psi^{h}_{2}} + \frac{\overline{\partial}f}{f}\wedge
\end{equation}
We now define the map 
\[
\Phi: H^{p,q}_{\psi^{h}_{1}}(U)\rightarrow H^{p,q}_{\psi^{h}_{2}}(U)
\]
as $\Phi([\alpha]):= [f\alpha]\in H^{p,q}_{\psi^{h}_{2}}(U)$
For two cohomologous $(p,q)$-forms $\alpha_{1},\alpha_{2}\in H^{p,q}_{\psi^{h}_{1}}(U)$, they differ by $\overline{\partial}_{\psi^{h}_{1}}\gamma$ where $\gamma\in\Omega^{p,q-1}(U)$.\\
Now
\[
\Phi([\alpha_{1}-\alpha_{2}]) = [f\overline{\partial}_{\psi^{h}_{1}}\gamma]
\]
\begin{align*}
\Rightarrow [f(\alpha_{1}-\alpha_{2})] &=[(\overline{\partial}_{\psi^{h}_{1}})(\gamma f) + \gamma\wedge \overline{\partial}f]\\ &=[( \overline{\partial}_{\psi^{h}_{2}} + \frac{\overline{\partial}f}{f}\wedge )(\gamma f) + \gamma\wedge \overline{\partial}f]
\end{align*}
\[
\Rightarrow [f\alpha_{1}] - [f\alpha_{2}] =[\overline{\partial}_{\psi^{h}_{2}}(\gamma f)] 
\]
since the classes $[f\alpha_{1}]$, $[f\alpha_{2}]$ differ by a $\overline{\partial}_{\psi^{h}_{2}}$-exact form, we conclude that the class $[f\alpha]\in H^{p,q}_{\psi^{h}_{2}}$ and the map $\Phi$ is a well defined homomorphism.\\
Now we check the bijectivity of this map.\\
Let us pick a cohomology class $[\alpha]\in H^{p,q}_{\psi^{h}_{1}}$ for which $\Phi([\alpha]) = [f\alpha]=0$. This means that the representative of this class is itself a $\overline{\partial}_{\psi^{h}_{2}}$-exact form. So there must be a $(p,q-1)$-form $\theta\in \Omega^{p,q-1}(U)$ such that
\begin{align*}
f\alpha &= \overline{\partial}_{\psi^{h}_{2}}\theta\\
&= \overline{\partial}\theta + \psi^{h}_{2}\wedge\theta\\
&= \overline{\partial}\theta + (\psi^{h}_{1}- \frac{\overline{\partial} f}{f})\wedge\theta\\
&= \overline{\partial}\theta +\psi^{h}_{1}\wedge\theta - \frac{\overline{\partial} f}{f}\wedge\theta\\
\end{align*}
\begin{align*}
\Rightarrow \alpha &= \frac{f\overline{\partial}\theta-\overline{\partial} f\wedge\theta}{f^{2}}+\frac{\psi^{h}_{1}\wedge\theta}{f}\\
&= \overline{\partial}(\frac{\theta}{f}) + \psi^{h}_{1}\wedge(\frac{\theta}{f})\\
&=\overline{\partial}_{\psi^{h}_{1}}(\frac{\theta}{f})\\
\end{align*}
saying that the representative $\alpha$ is itself a $\overline{\partial}_{\psi^{h}_{1}}$-exact form, so the cohomology class $[\alpha] = 0$ in the group $H^{p,q}_{\psi^{h}_{1}}(U)$. Thus $Ker(\Phi)= {0}$ and the map $\Phi$ is injective.\\
Now let us pick another class $[\beta]\in H^{p,q}_{\psi^{h}_{2}}(U)$ where the representative $\beta\in\Omega^{p,q}(U)$ is $\overline{\partial}_{\psi^{h}_{2}}$-closed.
This means that 
\[
\overline{\partial}_{\psi^{h}_{2}}\beta = 0
\]
\begin{equation}
    \Rightarrow \overline{\partial}\beta = - \psi^{h}_{2}\wedge\beta 
\end{equation}
\\
Again from (11) we see that 
\begin{equation}
    \psi^{h}_{1}\wedge\frac{\beta}{f} = \psi^{h}_{2}\wedge\frac{\beta}{f} + \overline{\partial}f\wedge\frac{\beta}{f^{2}}
\end{equation}
It follows that 
\[
\overline{\partial}_{\psi^{h}_{1}}(\frac{\beta}{f}) = \overline{\partial}(\frac{\beta}{f}) + \psi^{h}_{1}\wedge\frac{\beta}{f}
\]
\[
= \frac{f\overline{\partial}\beta - \overline{\partial}f\wedge\beta}{f^{2}} + \psi^{h}_{1}\wedge\frac{\beta}{f}
\]
Using (15), (16)
\[
= -\frac{\psi^{h}_{2}\wedge\beta}{f} - \overline{\partial}f\wedge\frac{\beta}{f^{2}} + \psi^{h}_{1}\wedge\frac{\beta}{f} = 0
\]
which means that the $\overline{\partial}_{\psi^{h}_{1}}$-closed form $\frac{\beta}{f}$ represents a class in the group $H^{p,q}_{\psi^{h}_{1}}(U)$. Thus the arbitrarily chosen class $[\beta]$ is the image of the map $\Phi$ i.e. $\Phi([\frac{\beta}{f}]) = [\alpha]$. So the map $\Phi$ is sujective as well.\\
We conclude that (locally) we have the isomorphism 
\begin{equation}
    H^{p,q}_{\psi^{h}_{1}}(U)\simeq H^{p,q}_{\psi^{h}_{2}}(U)
\end{equation}
This isomorphism also holds globally for two cohomologous $\eta_1,\eta_2$ differing by a globally $\overline{\partial}$-exact (logarithmic) form. \textbf{QED}\\ 
We are almost done with this proof here except a subtle remark that must be mentioned: If we look at equation (11) more closely, we see that the manifold $X^h$ \textit{compactified} with divisors with normal crossings. That is, the divisors $D_1 = \cdot\cdot\cdot=D_j=0$ where the function $f$ is vanishing. This is also called Hironaka compactification. This further restricts the algebraic twisting parameter. We may consider the algebraic twisting parameter on the projective variety $X$ (even if we forget about its \textbf{abelian group structure}) to be a logarithmic 1-form. Then, it is naturally implied that the projective variety is already Hironaka compactified with algebraic divisors with normal crossings.\\
We present another simpler proof. The idea is to use the principle that if the chain map is invertible between two chain complex, then the cohomologies of the two chain complexes are isomorphic and vice versa.\\

\textit{Proof 2}: Starting with the cohomologous analytic twisting parameters 
\begin{equation}
\psi^{h}_{1} = \psi^{h}_{2} + \frac{\overline{\partial}f}{f} =  \psi^{h}_{2} + \overline{\partial} \operatorname{log}f
 \end{equation}
Consider the function $r$ such that 
\[
e^{r} = f
\]
Define the map 
\begin{equation}
    T: \theta \mapsto e^{-r}\theta
\end{equation}
Now we compute 
\begin{align*}
    \overline{\partial}_{\psi_{2}^{h}}(T(\theta)) &= (\overline{\partial} + \psi^{h}_{2}\wedge)(e^{-r}\theta)\\
    &= \overline{\partial}(e^{-r}\theta) + \psi_{2}^{h}\wedge(e^{-r}\theta)\\
    &= -e^{-r}\overline{\partial}r \wedge \theta + e^{-r}\overline{\partial}\theta + e^{-r}\psi_{2}^{h}\wedge \theta\\
    &= -e^{-r}\overline{\partial}\operatorname{log}f \wedge \theta + e^{-r}\overline{\partial}\theta + e^{-r}\psi_{2}^{h}\wedge \theta\\
    &= e^{-r}(\overline{\partial}\theta + [\psi^{h}_{2}-\overline{\partial}\operatorname{log}f]\wedge \theta )\\
    &= e^{-r}(\overline{\partial}\theta + \psi^{h}_{1}\wedge \theta)\\ 
    &= T(\overline{\partial}_{\psi^{h}_{1}}\theta)
\end{align*}Thus the map (19) is a chain map. Under the assumption that $f$ is nowhere zero, we can invert this map. Then it is an isomorphism between the chain complexes associated with the analytic twisting parameters $\psi^{h}_{1,2}$ and hence the cohomologies are isomorphic. \textbf{QED}\\

However, since we are considering $X$ to be an abelian variety, under Hironaka compactification the algebraic twisting parameter itself then becomes logarithmic and hence we can use tools from logarithmic cohomology theory of algebraic varieties \cite{deligne2006equations}. So theorem 1 gave us an important hint that logarithmic structures must be equipped on the abelian variety $X$. So we briefly recall some facts and examples in the next subsection (the reader familiar with this technology may skip it).
\subsubsection{Logarithmic cohomology} For an algebro geometric introduction to logarithmic geometry, probably the most comprehensive reference is the book \cite{ogus2018lectures} and the beautifully written \cite{deligne2006equations}. Below, we are considering schemes to be honest. Yet, we shall write "variety" instead of "scheme". To keep things with the flow of the previous sections above, we are adapting this abuse of notation.\\
Now suppose $D\subset X$ is a closed subvariety of $X$ which is a smooth divisor over the ring $R$. In terms of local coordinates $(t_1,...,t_n)$ on $X$, the divisor $D$ is assumed to be defined by $t_1 =\cdot\cdot\cdot = t_l= 0$.\footnote{This is generally known as a log-structure being equipped to a variety\cite{ogus2018lectures}.} So the $\mathcal{O}_{X/S}$-module alluded to above is seen to be generated by the differentials $dt_{1},...,dt_n$ (i.e. $\Omega_{X/S}^{1}= \langle dt_1,...,dt_n \rangle$), in terms of these coordinate variables.\\
Define a \textit{variant} of this module which is generated as 
\begin{equation}
    \Omega_{X/S}^{1}(\operatorname{log} D) = \langle \frac{dt_1}{t_1},...,\frac{dt_l}{t_l},t_{l+1},...,t_n \rangle
\end{equation}
Clearly the sections of this sheaf are singular along the divisor $D$ and the singularities being of order one, the poles are said to be \textit{logarithmic}. Thus this is called the sheaf of logarithmic differentials. The higher order logarithmic forms are obtained from taking wegde product: $\Omega_{X}^{k}(\operatorname{log} D) = \bigwedge^{k}\Omega_{X}^{1}(\operatorname{log} D)$. Then we can write the logarithmic de Rham complex $(\Omega_{X/S}^{\bullet}(\operatorname{log}D),d)$:
    \begin{equation}
    0\overset{}{\rightarrow}  \Omega_{X/S}^{0}(\operatorname{log}D) \overset{d^{0}}{\rightarrow} \Omega_{X/S}^{1}(\operatorname{log}D)\overset{d^{1}}{\rightarrow} \Omega_{X/S}^{2}(\operatorname{log}D)\overset{d^{2}}{\rightarrow}.... \overset{d^{n-1}}{\rightarrow}\Omega_{X/S}^{n}(\operatorname{log}D)\overset{d^{n}}{\rightarrow}0 
\end{equation}
As defined in \cite{deligne2006equations}, the hypercohomology of the logarithmic de Rham complex is the de Rham cohomology of the complement $X-D$: 
\begin{equation}
    \operatorname{H}_{dR}^{i}(X-D) := \textbf{H}^{i}(X,\Omega_{X/S}^{\bullet}(\operatorname{log}D))
\end{equation}
For computational convenience, it is useful to introduce an $R$-linear sheaf homomorphism 
\begin{equation}
    \operatorname{Res}: \Omega_{X/S}^{1}(\operatorname{log} D) \longrightarrow \mathcal{O}_D = \bigoplus_{j}(i_j)_{*}\mathcal{O}_{D_{j}}
\end{equation}
where $D= \sum_j D_j$ with the embeddings $i_j: D_j\hookrightarrow X$, known as the residue map.
This map acts as 
\begin{equation}
    f_1 \frac{dt_1}{t_1}+ ... +f_l \frac{dt_l}{t_l}+ f_{l+1}dt_{l+1}+...+f_ndt_n \mapsto f_{1}|_{t_{1}=0} + .... + f_{l}|_{t_{l}=0}
\end{equation}
from which it is obvious that we are essentially modding out the K\"{a}hler differentials from the logarithmic differentials while applying the residue map, hence we can write 
\begin{equation}
    \frac{\Omega_{X/S}^{1}(\operatorname{log}D)}{\Omega_{X/S}^{1}}  \simeq \mathcal{O}_{D}
\end{equation}
Hence we can form the following short exact sequence\footnote{ This is a straight forward generalization of the following short exact sequence from complex analysis:
\begin{equation}
    0\longrightarrow \mathbb{C}\{\{z\}\}dz\longrightarrow \mathbb{C}\{\{z\}\}\frac{dz}{z}\overset{Res}{\longrightarrow}\mathbb{C}\longrightarrow 0 
\end{equation}
where the first map is from the ring of holomorphic differentials to the meromorphic ones. The residue acts as 
\[
\operatorname{Res}: f(z)\frac{dz}{z} \mapsto f(0)
\]
} 
\begin{equation}
    0 \longrightarrow \Omega_{X/S}^{1} \longrightarrow \Omega_{X/S}^{1}(\operatorname{log}D)\overset{Res}{\longrightarrow}\underbrace{\mathcal{O}_{D}}_{\bigoplus_{j}(i_j)_{*}\mathcal{O}_{D_{j}} = \Omega_{D}^{0}}\longrightarrow0
\end{equation}
This can be lifted to higher rank differentials 
\begin{equation}
    0 \longrightarrow \Omega_{X/S}^{i} \longrightarrow \Omega_{X/S}^{i}(\operatorname{log}D)\overset{Res}{\longrightarrow}\Omega_{D}^{i-1}\longrightarrow0
\end{equation}
More generally, from  the short exact sequence of the complexes 
\begin{equation}
    0 \longrightarrow \Omega_{X/S}^{\bullet} \longrightarrow \Omega_{X/S}^{\bullet}(\operatorname{log}D)\overset{Res}{\longrightarrow}\Omega_{D}^{i\bullet-1}\longrightarrow0
\end{equation}
we can get the long exact sequence of hypercohomologies 
\begin{equation}
    0 \longrightarrow \operatorname{H}^{1}(X) \longrightarrow \operatorname{H}^{1}(X-D)\longrightarrow \operatorname{H}^{0}(D)\longrightarrow \operatorname{H}^{2}(X)\longrightarrow\cdot\cdot\cdot
\end{equation}
\subsubsection{An example for curves}
Let us narrow down our exploration to the case of relative dimension 1, so the variety $X$ over the ring $R$ has geometrically integral curves as it's fibres and the divisor D is given by a section. \footnote{We refer to the simple calculation \cite{poonen2020algebraic} which is inspired from \cite{katz1973p}}So now divisor\footnote{really an effective Cartier divisor} is defined from a single equation $t=0$. Let $\mathcal{I}_D$ be the corresponding ideal sheaf, which is generated by $1/t$, so is isomorphic to the line bundle $\mathcal{O}(D) = \langle 1/t \rangle$. Since we also have that $\Omega_{X}^{1}(\operatorname{log}D) = \langle dt/t \rangle, \Omega_X ^{1} = \langle dt \rangle$, it is obvious that 
\begin{equation}
    \Omega_{X}^{1}(\operatorname{log} D) \simeq \Omega_{X}^{1} \otimes_{\mathcal{O}_X} \mathcal{I}_D \simeq \Omega_{X}^{1} \otimes_{\mathcal{O}_X} \mathcal{O}(D) = : \Omega_{X}(D)
\end{equation}
Thus the logarithmic differentials on a curve are essentially obtained from twisting them with the line bundle associated to an effective Divisor of the curve.\\
Since $D$ is a section of the variety morphism $X \longrightarrow \operatorname{Spec}(R)$\footnote{Being a closed immersion of $\operatorname{Spec}(R)$ into the variety $X$, the induced subvariety is canonically isomorphic to $\operatorname{Spec}(R)$. So the section group $\operatorname{H}^{0}(D) = \operatorname{H}^{0}(\operatorname{Spec}(R))\simeq R$}, there is an isomorphism $\operatorname{H}^{0}(D)\simeq \operatorname{H}^{2}(X)\simeq R$. Thus from the long exact sequences of hypercohomologies (52), we get another isomorphism $\operatorname{H}^{1}(X)\simeq \operatorname{H}^{1}(X-D)$, or $\mathbf{H}^{1}(X, \mathcal{O}_{X}\overset{d}{\rightarrow}\Omega_{X}^{1}) \simeq \mathbf{H}^{1}(X, \mathcal{O}_{X}\overset{d}{\rightarrow}\Omega_{X}^{1}(D))$. So the inclusion of the complex $\mathcal{O}_{X}\overset{d}{\rightarrow}\Omega_{X}^{1}$ into $\mathcal{O}_{X}\overset{d}{\rightarrow}\Omega_{X}^{1}(D)$ is a quasi isomorphism.\\
Now if we further twist the sheaf of K\"{a}hler differentials $\Omega_{X}^{1}$ to $\Omega_{X}^{1}\otimes \mathcal{O}_{X}(2D) =: \Omega_{X}(2D) = \langle dt/t^2\rangle$, that has sections: meromorphic differentials with poles of order upto 2, from the following isomorphism 
\begin{equation}
    \frac{\mathcal{O}_{X}(D)}{\mathcal{O}_{X}} \simeq \frac{\Omega_{X}(2D)}{\Omega_{X}(D)}
\end{equation}
\[
\frac{1}{t}\mapsto -\frac{dt}{t^2}
\]
we see that the complex $d: \mathcal{O}_{X}\longrightarrow \Omega_{X}(D)$ can be quasi-isomorphically embedded into the complex $d: \mathcal{O}_{X}(D)\longrightarrow \Omega_{X}(2D)$, as the complex (54) can be seen as the quotient complex of these two complexes. Combining these two isomorphisms, we see that the cohomology reduces to 
\begin{equation}
    \operatorname{H}^{1}(X) = \textbf{H}^{1}(d:\mathcal{O}_{X}(D)\longrightarrow \Omega_{X}(2D))
\end{equation}
Which was shown to be $\operatorname{H}^{1}(X) \simeq \Gamma(X,\Omega(2D))$ for an ellpitic curve(\cite{poonen2020algebraic}).
\section{Weight sheaf} Here we mention another justification of our claim that there must logarithmic structure on the variety $X$ so that we can define algebraic twisting parameter on it in a way compatible (via the GAGA theorem)  with the analytic twisting parameter.\\
 To define the twisted Dolbeault cohomology, in \cite{meng2020morse} they define it using the soft resolution of the so called weight-$\eta$ sheaf $\mathcal{O}_{X^h,\eta}$ which is defined to  be the kernel of the sheaf homomorphism 
 \begin{equation}
     \overline{\partial}_{\eta}: \Omega_{X^h}^{p,0}\rightarrow \Omega_{X^h}^{p,1}
 \end{equation}
 where $\eta$ is the analytic twisting parameter i.e. $\eta\in \Omega_{X^h}^{0,1}$. Using the Grothendieck-Poincar\'{e} lemma the weight-$\eta$ sheaf $\mathcal{O}_{X^h,\eta}$ can be shown to be a locally free sheaf of rank 1 and then 
 \begin{equation}
     \Omega_{X^h,\eta}^{p} = \Omega_{X^h}^p \otimes_{\mathcal{O}_{X^h}} \mathcal{O}_{X^h,\eta} 
 \end{equation}
 \[
 = \Omega_{X^h}^p \otimes_{\mathcal{O}_{X^h}} \Omega_{X^h,\eta}^0
 \]
  As was noted in \cite{meng2020morse,meng2023morse} the sheaf $\Omega_{X^h, \eta}^{p}$ can be written locally as $e^{-u}\Omega_{X^h}^{p}$ (and $\mathcal{O}_{X^h,\eta} = e^{-u}\mathcal{O}_{X^{h}}$) using Grothendieck-Poincar\'{e} lemma \cite{serre1953faisceaux}. For a smooth complex valued 1-form on $X^h$ $\theta$, the generalized weight sheaf is then the $\theta$ weight sheaf defined as the kernel of the sheaf homomorphism
\begin{equation}
    d_{\theta}: \Omega_{X^h}^{0}\longrightarrow \Omega_{X^h}^{1}
\end{equation}
where $d_{\theta}\alpha = d\alpha + \theta \wedge \alpha$, that (following \cite{meng2020morse}) we denote as $\underline{\mathbb{C}}_{X^h,\theta}$. As before, because the 1-form $\theta$ can be written as the differential of a complex valued smooth function, $\theta = du$, we can write the weight sheaf as

\begin{equation}
\underline{\mathbb{C}}_{X^h,\theta} = \underline{\mathbb{C}}e^{-u}
\end{equation}
From the decomposition 
\begin{equation}
    \Omega^{1}_{X^h} \otimes \mathbb{C} =  \Omega^{1,0}_{X^h} \oplus \Omega_{X^{h}}^{0,1} 
\end{equation}
the complex 1-form can be split 
\begin{equation}
    \theta = \eta + \bar{\zeta}
\end{equation}
 where $\eta,\zeta \in \Omega_{X^h}^{1,0}$. Then $\mathcal{O}_{X^h,\eta}\cap \overline{\mathcal{O}_{X^h,\zeta}} = \underline{\mathbb{C}}_{X^{h},\theta}$.\\

 Now the sheaf (35) can be seen as the zeroth cohomology with coefficients in the (analytic) vector bundle (in this case line bundle with flat/integrable connection\cite{ornea2009morse})
 \begin{equation}
     \text{H}_{dR}^{0}(X^h,(\mathcal{V}^h,\nabla^h)) = \operatorname{ker}(\nabla^h) = \underline{\mathbb{C}}e^{-u}
 \end{equation}
 for some analytic connection $\nabla^h$ (somehow depending/corresponding to the form $\theta$) However, for the algebraic variety $X$ we have that 
 \begin{equation}
     \text{H}_{dR}^{0}(X^h,(\mathcal{V},\nabla)) = \operatorname{ker}(\nabla) = 0 
 \end{equation}
 As a final remark, we want to stress the fact that the weight $\eta$-sheaf $\mathcal{O}_{X,\eta}$ is essentially a line bundle, that we may denote as $L_{\eta}$ (being a locally free sheaf of rank 1) defined via the transition functions $e^{u_{i}-u_{j}}: U_{i}\cap U_{j}\longrightarrow \operatorname{GL}(1,\mathbb{C})\simeq \mathbb{C}^{\times}$ on non empty overlapping charts, where $\eta = \overline{\partial}u_{i}$ locally on $U_{i}\subseteq X$ via the Grothendieck Poincar\'{e} lemma. The Dolbeault-Morse-Novikov cohomology is the cohomology valued in this line bundle for which it must be equipped with a flat connection which in this case is the twisted derivative $\overline{\partial}_{\eta}$. This implies that the holomorphic structure on $L_{\eta}$ from the nilpotency. 

 \subsection{Recap on Connections} We review some aspects of connections\footnote{see for example \cite{dupuyintroduction} for a nice introduction} relative varieties for generality. Starting with 
 \begin{definition}
A variety $S$ and an $S$-variety $X$ i.e. there is a variety morphism $\pi: X \longrightarrow S$ and a quasi-coherent sheaf of $\mathcal{O}_X$-module $\mathcal{E}$, the $\pi^{-1}\mathcal{O}_{S}$-linear morphism
 \begin{equation}
     \nabla: \mathcal{E}\longrightarrow \mathcal{E}\otimes_{\mathcal{O}_{X}}\Omega_{X/S}^1
 \end{equation}
 is defined to be (an algebraic) connection on the sheaf $\mathcal{E}$.
 \end{definition}
 For any section $f\in \Gamma(X,\mathcal{O}_X)$ and $e\in \Gamma(X,\mathcal{E})$ it satisfies 
 \begin{equation}
     \nabla(fe) = e \otimes df + f \nabla e
 \end{equation}
 If the sheaf is further restricted to the class of locally free sheaf of finite rank of $\mathcal{O}_X$-modules, then we say that the pair $(\mathcal{E},\nabla)$ is a vector bundle with connection.\\
 We look at the example of a trivial bundle on $X$
 \begin{equation}
     \mathcal{E} = \bigoplus_{i=1}^{r}\mathcal{O}_{X}e_{i}
 \end{equation}
 of rank $r$ that is said to be trivialized by the sections (generators) $(e_1,...,e_r)$. Then we take the global sections of the K\"{a}hler differential sheaf $\Omega_{X/S}$, $\omega_{ij}\in \Gamma(X,\Omega_{X/S}^{1})$, we can write the following \textit{change of basis}
 \begin{equation}
     \nabla e_j = \sum_{i=1}^{r}e_i \otimes \omega_{ij}
 \end{equation}
 where $j=1,...,r$. Just as in classical differential geometry, the matrix of the K\"{a}hler differentials $\Omega= [\omega_{ij}]$ and \textit{symbolically} we can write 
 \begin{equation}
     \nabla = d + \Omega
 \end{equation}
 Since any section of $\mathcal{E}$ is an $\mathcal{O}_X$-linear combination of the generators i.e. $e = \sum f_{j}e_{j}$, the connection then acts on it as 
 \begin{equation}
     \nabla e = \sum_{i=1}^{r}e_{i}\otimes (df_{i} + \sum_{j=1}^{r}\omega_{ij}f_{j})
 \end{equation}
 A section acting on which the connection gives zero is called a horizontal section with respect to that connection: $e \in \Gamma(\mathcal{E}): \nabla e=0$. Thus, horizontal sections correspond to solutions $(f_1,...,f_r)$ of the linear differential equation
 \begin{equation}
     df_i + \sum_{j=1}^{r}\omega_{ij}f_{j} = 0
 \end{equation}
 Let's call it the associated differential equation of the connection.\\
 Now consider the Hironaka compactification of the variety $X$, which we denote as $\overline{X}$ that is we embed it inside a larger space $j: X \hookrightarrow \overline{X}$ such that $\overline{X}-X = D$ is a divisor with normal crossings. Then we have (from \cite{deligne2006equations}) the following 
 \begin{definition}
 A connection $\nabla$ is called \textbf{regular} if near every point of $D$, the solutions to the associated differential equations grow at most polynomially along paths approaching $D$. That is, if the connection $\nabla: \mathcal{E}\longrightarrow \mathcal{E}\otimes_{\mathcal{O}_{X}}\Omega_{X}^{1}$ extends to the connection 
 \begin{equation}
     \overline{\nabla}: \mathcal{E}\longrightarrow \mathcal{E}\otimes_{\mathcal{O}_{\overline{X}}}\Omega_{\overline{X}}^{1}(\operatorname{log} D)
 \end{equation}
 \end{definition}
 Which is to say that the connection has logarithmic singularities at worst. This is called Deligne extension of the connection.\\
 \textbf{Example}: Consider the connection $\nabla = d - \frac{\lambda}{t}dt$ on the (open) variety $X = \mathbb{A}_{\mathbb{C}}^{1}- \{0\} = \operatorname{Spec}(\mathbb{C}[t,t^{-1}])$ with the trivial line bundle $\mathcal{O}_{X} = \mathcal{E}$, where $\lambda\in \mathbb{C}$. The associated differential equation for a horizontal section $s$ i.e. $\nabla e =0$:
 \[
 ds = \frac{\lambda}{t}s dt
 \]
 whose solution is $s= kt^{\lambda}$ for some $k\in \mathbb{C}^{\times}$. This is connection has a singularity at $t=0$. However, as can be seen that as $t\longrightarrow0$, the solution $t^{\lambda}$ grows polynomially. So the singular point is a regular one. A similar analysis can be used to show that $\infty$ is also a singularity. We can compactify $X$ to get the complex projective space $\mathbb{CP}^{1}$.\\
 The reason for recalling regular connections on algebraic varieties with Hironaka compactification should be clear in the next section where we mention an important result by Pierre Deligne.
\subsection{A theorem of Deligne}
This theorem gave us another strong hint why we must equip the variety $X$ with logarithmic structure. We won't describe it's proof but only the statement of course. To set the stage, let us recall some notions.\\ 
Let $X$ be a smooth algebraic variety of dimension $n$ over the complex numbers. We define the sheaf of algebraic $p$-forms $\Omega_{X}^{p}$ as the $p$-times \textit{wedge product} of the sheaf of K\"{a}hler differential forms $\Omega_{X}^{1}$.
We can then consider the resolution of the constant sheaf $\underline{\mathbb{C}}$ 
\begin{equation}
    \underline{\mathbb{C}}\longrightarrow\Omega_{X}^{0}\longrightarrow\Omega_{X}^{1}\longrightarrow\Omega_{X}^{2}\longrightarrow \cdot\cdot\cdot
\end{equation}
and the hypercohomology of this complex is the algebraic de Rham cohomology \cite{grothendieck1966rham}
\begin{equation}
    \text{H}_{dR}^{*}(X) := \textbf{H}^{*}(\underline{\mathbb{C}}\longrightarrow\Omega_{X}^{0}\longrightarrow\Omega_{X}^{1}\longrightarrow\Omega_{X}^{2}\longrightarrow \cdot\cdot\cdot)
\end{equation}
We take into account the fact that for a complex of \textit{fine} sheaves, the cohomology of the complex of sections of the sheaves is equivalent to the hypercohomology of the complex of those sheaves in the case of the sheaf of the analytic differential forms of the associated analytic variety $X^{h}$ (see section 2) for which we have the complex 
\begin{equation}
    \underline{\mathbb{C}}\longrightarrow\Omega_{X^{h}}^{0}\longrightarrow\Omega_{X^{h}}^{1}\longrightarrow\Omega_{X^{h}}^{2}\longrightarrow \cdot\cdot\cdot
\end{equation}
and the corresponding complex to the underlying smooth manifold
\begin{equation}
\underline{\mathbb{R}}\longrightarrow\Omega_{X^{\infty}}^{0}\longrightarrow\Omega_{X^{\infty}}^{1}\longrightarrow\Omega_{X^{\infty}}^{2}\longrightarrow \cdot\cdot\cdot
\end{equation}
More precisely, the hypercohomology of this (resolution of $\underline{\mathbb{R}}$) complex is isomorphic to the cohomology of the complex 
\begin{equation}
\Gamma(X^{\infty},\Omega_{X^{\infty}}^{0})\longrightarrow\Gamma(X^{\infty},\Omega_{X^{\infty}}^{1})\longrightarrow\Gamma(X^{\infty},\Omega_{X^{\infty}}^{2})\longrightarrow\cdot\cdot\cdot
\end{equation}
since $\Omega_{X^{\infty}}^{p}$ are \textit{fine} sheaves.\\
Then we get a chain of isomorphisms 
\begin{equation}
    \text{H}_{dR}^{*}(X)\longrightarrow\text{H}_{dR}^{*}(X^{h})\longrightarrow\text{H}_{dR}^{*}(X^{\infty})
\end{equation}
where the first isomorphism is highly non trivial \cite{grothendieck1966rham}. However all these cohomologies are computed in constant coefficients. These can be powerfully generalized to non constant coefficient cases namely algebraic bundle valued cases.\\
Let $\mathcal{V}$ be an algebraic vector bundle (locally free sheaf of finite rank) on $X$ equipped with an integrable connection (i.e. \textit{flat})
\begin{equation}
    \nabla: \mathcal{V}\longrightarrow \Omega_{X}^{1}\otimes \mathcal{V}
\end{equation}
Using such a connection we may write the complex 
\begin{equation}
    \mathcal{V}\longrightarrow\Omega_{X}^{1}\otimes \mathcal{V}\longrightarrow\Omega_{X}^{2}\otimes \mathcal{V}\longrightarrow\cdot\cdot\cdot
\end{equation}
Thus we have the non constant coefficient cohomology as the hypercohomology of this complex 
\begin{equation}
    \text{H}_{dR}^{*}(X,(\mathcal{V},\nabla)) := \mathbf{H}^{*}(X,\mathcal{V}\longrightarrow\Omega_{X}^{1}\otimes \mathcal{V}\longrightarrow\Omega_{X}^{2}\otimes \mathcal{V}\longrightarrow\cdot\cdot\cdot)
\end{equation}
and the analytic counterparts 
\begin{equation}
    \mathcal{V}\longrightarrow\Omega_{X^{h}}^{1}\otimes \mathcal{V}\longrightarrow\Omega_{X^{h}}^{2}\otimes \mathcal{V}\longrightarrow\cdot\cdot\cdot
\end{equation}
\begin{equation}
    \text{H}_{dR}^{*}(X^h,(\mathcal{V},\nabla)) := \mathbf{H}^{*}(X^h,\mathcal{V}\longrightarrow\Omega_{X^h}^{1}\otimes \mathcal{V}\longrightarrow\Omega_{X^h}^{2}\otimes \mathcal{V}\longrightarrow\cdot\cdot\cdot)
\end{equation}
However, unlike (53), even though we do get a map 
\begin{equation}
     \text{H}_{dR}^{*}(X,(\mathcal{V},\nabla)) \longrightarrow   \text{H}_{dR}^{*}(X^h,(\mathcal{V},\nabla)) 
\end{equation}
but it's not an isomorphism. However, if we take another proper and smooth variety $\bar{X}$ such that $\bar{X}-X$ is a divisor with normal crossings we get a Hironaka compactification of the variety $X$. Then we can extend the algebraic vector bundle $(\mathcal{V},\nabla)$ to $(\bar{\mathcal{V}},\bar{\nabla})$ such that the extended connection $\bar{\nabla}$ has logarithmic poles along the divisor $\bar{X}-X$ with normal crossings. Then GAGA (see section 2) can be applied to the logarithmic cohomologies 
\begin{equation}
    \text{H}_{dR,log}^{*}(\bar{X},(\bar{\mathcal{V}},\bar{\nabla})) \longrightarrow \text{H}_{dR,log}^{*}(\bar{X^h},(\bar{\mathcal{V}}^h,\bar{\nabla}^h)) 
\end{equation}
Then these meromorphic de Rham cohomologies can be compared to the algebraic cohomologies with essential singularities at infinities, via the ramifications. This was precisely the strategy in \cite{deligne2006equations} and hence Deligne deduced the theorem\\ 

\begin{theorem}
For a regular connection\footnote{The connection has at worst regular singularities} $\nabla$, the cohomologies $\text{H}_{dR}^{*}(X,(\mathcal{V},\nabla))$ and $\text{H}_{dR}^{*}(X^h,(\mathcal{V}^h,\nabla^h))$ are isomorphic.\footnote{For a brief survey see \cite{andre2004comparison}}
\end{theorem}
The map (59) not being an isomorphism is precisely the reason why the algebraic twisted cohomology (if it exists) cannot be directly related to the analytic twisted cohomology via the GAGA theorem. And Deligne's theorem seems to be a viable solution to this, which imposes the condition of logarithmic structure on the variety X. 
\section{Algebraic Twist}
We shall relax the abelian group structure condition for the moment. Instead we may write "if it exists" i.e. if the twisting parameter globally exists on any variety. We shall try to define the algebraic twisted cohomology in cases as general as possible. \\ 
 Let $X$ be an $S$-variety with a locally free sheaf $\mathcal{E}$ (of finite rank) with a family of (unique) morphisms of abelian sheaves 
 \begin{equation}
     \nabla^{i}: \mathcal{E}\otimes_{\mathcal{O}_{X}} \Omega^{i}_{X/S} \longrightarrow \mathcal{E}\otimes_{\mathcal{O}_{X}} \Omega^{i+1}_{X/S}
 \end{equation}
 such that 
 \begin{equation}
     \nabla^{i}(e \otimes \omega) = e \otimes d\omega + (-1)^{i}\nabla(e)\wedge \omega
 \end{equation}
 Assuming integrability of the connection, we can write the complex 
 \begin{equation}
     \mathcal{E}\otimes_{\mathcal{O}_{X}} \Omega^{\bullet}_{X/S}: 0\longrightarrow \mathcal{E} \xrightarrow{\nabla}\mathcal{E}\otimes_{\mathcal{O}_{X}} \Omega^{1}_{X/S}\xrightarrow{\nabla}\mathcal{E}\otimes_{\mathcal{O}_{X}} \Omega^{2}_{X/S}\longrightarrow \cdot \cdot \cdot
 \end{equation}
 And for the push-forward of the complex $\pi_{*}(\mathcal{E}\otimes_{\mathcal{O}_{X}} \Omega^{\bullet}_{X/S})$, the hypercohomology of which is the cohomology valued/with coefficients in the vector bundle $(\mathcal{E},\nabla)$
 \begin{equation}
     \operatorname{H}_{dR}^{n}(X/S, (\mathcal{E},\nabla)) := \textbf{R}^{n}\pi_{*}(\mathcal{E}\otimes_{\mathcal{O}_{X}} \Omega^{\bullet}_{X/S})
 \end{equation}
 \begin{definition}
 Suppose further that for each globally defined (if they exist) K\"{a}hler differentials on $X/S$, $\omega \in \Gamma(X/S,\Omega_{X/S}^{1})$ we have a trivial line bundle $\mathcal{L}$ whose \textbf{regular} connection is (symbolically) $\nabla = d + \omega:\mathcal{L}\longrightarrow  \mathcal{L}\otimes_{\mathcal{O}_{X}} \Omega_{X/S}^{1}$. Then we define the \textbf{algebraic Twisted }\footnote{or algebraic Morse Novikov cohomology} cohomology as 
 \begin{equation}
     \operatorname{H}_{\omega}^{n}(X/S) :=  \operatorname{H}^{n}(X/S, (\mathcal{L,\nabla})) = \textbf{R}^{n}\pi_{*}(\mathcal{L}\otimes_{\mathcal{O}_{X}} \Omega^{\bullet}_{X/S})
 \end{equation}
 \end{definition}
 Let us compute some simple examples in the following subsection.
 \subsubsection{Example : Affine case}
Let $X$ be a smooth affine algebraic variety over some algebraically closed field $k$ of characteristic zero and let $\omega\in\Omega_{X}^{1}$ be a closed K\"{a}hler differential i.e. $d\omega = 0$. The trivial line bundle over $X$ is an $\mathcal{O}_X$-module of rank 1: $\mathcal{L} = \mathcal{O}_{X}\cdot e$ (thus its sections look like $f\cdot e \in \mathcal{L}$, for any $f\in \mathcal{O}_X$). The connection is the following map 
\begin{equation}
    \nabla: \underbrace{\mathcal{L}}_{=\mathcal{O}_X \cdot e}\longrightarrow \mathcal{L}\otimes_{\mathcal{O}_X}\Omega_X ^{1}\simeq \Omega_{X}^{1}\cdot e
\end{equation}
such that $\nabla(e)= -\omega \otimes e$ (which we recall from section 4). Thus the connection is the twisted differential $d_{\omega}:= d + \omega\wedge$. The closedness of $d_\omega$ follows from the flatness of the line bundle and vice versa. Indeed, the cohomology valued in a flat line bundle is isomorphic to the twisted de Rham cohomology with a deformed de Rham chain complex in which the exterior derivative is deformed with a closed 1-form which is precisely the connection 1-form of the connection of the flat line bundle. For a globally defined 1-form, the line bundle must be trivialized.  \\
Recall what the algebraic de Rham cohomologies are for a punctured elliptic curve over $k$
\[
X = \operatorname{Spec} \frac{k[x,y]}{\langle y^2 - f(x)\rangle} = \operatorname{Spec}A
\]
\[
f(x) = 4x^3 - g_{2}x - g_3
\]
 such that $g_{2}^{3}-27g_{3}^{2}\neq 0$. The K\"{a}hler differentials are 
 \[
 \Omega^{1}_{X/k} = \frac{A\cdot dx \oplus A\cdot dy}{\langle d(y^2 - f(x)) \rangle} = [A \oplus \frac{1}{y}A]dx 
 \]
Indeed the sheaf $\Omega_{X/k}^{1}$ is trivialized by the section $\alpha = dx/y$.\footnote{which is closed as can be seen from $d(\frac{dx}{y}) = -\frac{dy\wedge dx}{y^2} = -\frac{f'(x)}{2y}dx \wedge dx = 0$} Then the algebraic de Rham cohomologies are 
\[
\operatorname{H}^{0}(X/k) = k
\]
\[ \operatorname{H}^{1}(X/k) = k[\alpha]\oplus k[x\alpha]
\]
We take this as the twisting parameter\footnote{ the \textit{twisted derivative} $d_{\omega}:= d + \omega\wedge$ is a deformation of the usual exterior derivative of the de Rham complex, here the closed 1-form $\omega$ is called a twisting parameter} i.e. $\omega = dx/y$. Thus the twisted de Rham  complex is 
\begin{equation}
    0\longrightarrow A\overset{d_{\omega}^{0}}{\longrightarrow}\Omega_{X/k}^{1}\overset{d_{\omega}^{1}}{\longrightarrow}0
\end{equation}
Hence the zeroth cohomology of this complex can be computed from the kernel of the first map while the first cohomology from the cokernel of this map:\\
\begin{enumerate}
    \item $\operatorname{H}_{\omega}^{0}(X/k) = \operatorname{Ker}(A\overset{d_{\omega}^{0}}{\longrightarrow}\Omega_{X/k}^{1})$:\\
    However, it's obvious that for any section $A$ to vanish algebraically uder the action of the twisted derivative ($df + f\omega =0\Rightarrow f=0$), it must be zero. If not, it must be a transcendental functions. Hence zero is the only element in this cohomology.
    \[
    \operatorname{H}_{\omega}^{0}(X/k)=0
    \]
    \item $\operatorname{H}_{\omega}^{1}(X/k) = \operatorname{Coker}(A\overset{d_{\omega}^{0}}{\longrightarrow}\Omega_{X/k}^{1}) = \Omega^{1}_{X/k}/\operatorname{Im}(d_{\omega}^{0})$:\\
    It turns out that the global 1-form $dx/y$ is not contained in the image of the twisted derivative $\operatorname{Im}(d_{\omega}^{0})$. Indeed, if for any section $f\in A$, if we assume that $df + f\omega = dx/y (=\omega)$, the section $f$ must be transcendental. So $\operatorname{Im}(d_{\omega}^{0})\simeq A\cdot dx$. Hence
    \[
    \operatorname{H}^{1}_{\omega}(X/k) = \frac{\Omega_{X/k}^{1}}{\operatorname{Im}(d_{\omega}^{0})} = \frac{[A\oplus \frac{1}{y}A]\cdot dx}{A\cdot dx} = k\cdot [\frac{dx}y]\simeq k
    \]
\end{enumerate} 

\textbf{Remarks}: Comparing with the algebraic de Rham, we see the algebraic consequence of this twisting procedure. In the zeroth cohomology, all the constants have vanished. While the first cohomology has been reduced from two dimensional to one dimensional.

\subsection{Logarithmic twist}
 Regularity condition on the connection implies that we can extend this cohomology to the \textbf{logarithmic de Rham complex} \cite{deligne2006equations} on the Hironaka compactification of the variety, $\overline{X}$ (the line bundle and the connection extend as well): 
 \begin{equation}
   \operatorname{H}_{\tilde{\omega}}^{n}(\overline{X}) :=   \operatorname{H}^{n}(\overline{X}, (\overline{\mathcal{L}},\overline{\nabla})) = \textbf{R}^{n}\pi_{*}(\overline{\mathcal{L}}\otimes_{\mathcal{O}_{\overline{X}}} \Omega^{\bullet}_{\overline{X}}(\operatorname{log} D))
 \end{equation}
 where the one form $\tilde{\omega}$ is logarithmic. As a consequence, we can use this logarithmically extended cohomology for computing the algebraic Morse Novikov cohomology since for any Hironaka compactification $\overline{X}$ of a variety $X$, the embedding 
 \begin{equation}
     j: X \hookrightarrow \overline{X}
 \end{equation}
induces a quasi-isomorphism of the corresponding chain complexes 
\begin{equation}
   \overline{\mathcal{L}} \otimes_{\mathcal{O}_{\overline{X}}} \Omega^{\bullet}_{\overline{X}}(\operatorname{log} D) \longrightarrow j_{*}(\mathcal{L} \otimes_{\mathcal{O}_{X}}\Omega_{X}^{\bullet})
\end{equation}
for which we have isomorphic cohomologies 
\begin{equation}
    \operatorname{H}^{k}_{\omega}(X/S) \simeq   \operatorname{H}_{\tilde{\omega}}^{n}(\overline{X})
\end{equation}
We stress the point that the Hironaka compactification may be applied to a non compact smooth open variety. If the variety $X$ already contains divisors with normal crossings at its boundary (i.e. its compact in the sense of Hironaka compactification), then
from the sequence (51), we may write 

\[
\begin{tikzcd}
  \vdots & \vdots & \vdots &\\
  0\longrightarrow\mathcal{L}\otimes_{\mathcal{O}_{X}}\Omega_{X}^{k+1}  \arrow{r}{} \arrow{u}{\nabla} & \mathcal{L}\otimes_{\mathcal{O}_{X}}\Omega_{X}^{k+1}(\operatorname{log}(D)) \arrow{r}{d} \arrow{u}{\nabla} & \mathcal{L}\otimes_{\mathcal{O}_{X}}\Omega_{D}^{k} \arrow{r}{} \arrow{u}{\nabla} & 0\\
  0\longrightarrow \mathcal{L}\otimes_{\mathcal{O}_{X}}\Omega_{X}^{k} \arrow{r}{} \arrow{u}{\nabla} & \mathcal{L}\otimes_{\mathcal{O}_{X}}\Omega_{X}^{k}(\operatorname{log}(D)) \arrow{r}{} \arrow{u}{\nabla} & \mathcal{L}\otimes_{\mathcal{O}_{X}}\Omega_{D}^{k-1} \arrow{r}{} \arrow{u}{\nabla} & 0\\
  0\longrightarrow  \mathcal{L}\otimes_{\mathcal{O}_{X}}\Omega_{X}^{k-1}  \arrow{r}{} \arrow{u}{\nabla} & \mathcal{L}\otimes_{\mathcal{O}_{X}}\Omega_{X}^{k-1} \arrow{r}{} \arrow{u}{\nabla}(\operatorname{log}(D)) & \mathcal{L}\otimes_{\mathcal{O}_{X}}\Omega_{D}^{k-2} \arrow{r}{} \arrow{u}{\nabla} & 0\\
  \vdots\arrow{u} & \vdots\arrow{u} & \vdots\arrow{u}
\end{tikzcd}
\]
As an example case, we can extend the discussion in section 3.1.1. The trivial line bundle (or the invertible sheaf) $\mathcal{L}$ is the structure sheaf $\mathcal{O}_X$ itself. So we have 
\begin{equation}
    \mathcal{L}\otimes_{\mathcal{O}_X}\Omega_X(\operatorname{log}D) \simeq \mathcal{O}_X\otimes_{\mathcal{O}_X}\Omega_X(\operatorname{log}D) \simeq \Omega_X(\operatorname{log}D)
\end{equation}
Hence just as in section 3.1.1, the cohomology is precisely the same.

\subsubsection{Example: Non-affine case} We examine the case of an elliptic curve over $\mathbb{C}$: $E= \mathbb{C}/\Lambda$, which is a projective variety, whose Hironaka compactification is understood from the divisor with normal crossings $D= \{p_1 =\cdot\cdot\cdot=p_{n} = 0 \}$ where $p_1,...,p_n$ are points from $E$. We take the twisting parameter in the present case to be a global section $\omega \in \operatorname{H}^{0}(E,\Omega_{E}^{1}(\operatorname{log}D))$, where the sheaf $\Omega_{E}^{1}(\operatorname{log}D)$ is computed from the short exact sequence (49).\\
The variety being projective, the structure sheaf $\mathcal{O}_{E}$ has constants as its global sections i.e. $\Gamma(E,\mathcal{O}_{E})\simeq \mathbb{C}$. The connection of this sheaf is 
\begin{equation}
    \nabla: \mathcal{O}_{E}\longrightarrow \mathcal{O}_{E}\otimes\Omega_{E}^{1}(\operatorname{log}D)\simeq \Omega_{E}^{1}(\operatorname{log}D) 
\end{equation}
So the zeroth cohomology of the two-term complex 
\begin{equation}
    \mathcal{O}_{E}\overset{\nabla}{\longrightarrow} \Omega_{E}^{1}(\operatorname{log}D)
\end{equation}
is zero: $\operatorname{H}_{log,\omega}^{0}(E) =0$ as can be seen from the fact that $\nabla(f) = f\omega$, since $f \in \mathbb{C}$.\\
For computational convenience, let us briefly review hypercohomology using spectral sequence for the specialized case of curves which relevant to our case.\\
Recall that for any curve $X$ the cochain complex is a two-term cohomology 
\[
\mathcal{K}^{\bullet}: K^{0}\overset{d}{\longrightarrow}K^{1}
\]
for which the cohomology sheaves are 
\[
\mathcal{H}^{0} = \operatorname{Ker(d:K^0 \longrightarrow K^1)}
\]
\[
\mathcal{H}^{1} = \operatorname{Coker(d:K^0 \longrightarrow K^1)}
\]
\[
\mathcal{H}^{q} = 0
\]
for $q\geq 2$. Then the hypercohomologies are 
\begin{equation}
    \mathbb{H}^{0}(X,\mathcal{K}^{\bullet}) = \operatorname{H}^{0}(X,\mathcal{H}^{0})
\end{equation}
\begin{equation}
    \mathbb{H}^{1}(X,\mathcal{K}^{\bullet}) = \operatorname{H}^{1}(X,\mathcal{H}^{0}) \oplus \operatorname{H}^{0}(X,\mathcal{H}^{1})
\end{equation}
\begin{equation}
    \mathbb{H}^{2}(X,\mathcal{K}^{\bullet}) = \operatorname{H}(X,\mathcal{H}^1)
\end{equation}
We can reinterpret them in terms of the spectral sequence page 
\begin{center}
\begin{tikzpicture}
  \draw (0,0) rectangle (4,4);
  \draw (2,0) -- (2,4);
  \draw (0,2) -- (4,2);

  \node at (1,3) {$E_{1}^{0,0} $};
  \node at (3,3) {$E_{1}^{0,1}$};
  \node at (1,1) {$E_{1}^{1,0}$};
  \node at (3,1) {$E_{1}^{1,1}$};
\end{tikzpicture}
\end{center}
where the terms are 
\begin{enumerate}
    \item $E_{1}^{0,0} = \operatorname{H}^{0}(X,K^{0})$
    \item $E_{1}^{0,1} = \operatorname{H}^{1}(X,K^{0})$
    \item $E_{1}^{1,0} = \operatorname{H}^{0}(X,K^{1})$
    \item $E_{1}^{1,1} = \operatorname{H}^{1}(X,K^{1})$ 
\end{enumerate}
The zeroth hypercohomology is 
\begin{equation}
    \mathbb{H}^{0}(X,\mathcal{K}^{\bullet}) = \operatorname{Ker}(E_{1}^{0,0}\longrightarrow E_{1}^{1,0}) = \operatorname{Ker}(\operatorname{H}^{0}(X,K^{0})\longrightarrow \operatorname{H}^{0}(X,K^1)) 
\end{equation}
The first hypercohomology can be computed from the totalized complex of the spectral sequence page 
 that consists of the terms 
 \[
 D^{0} = E_{1}^{0,0}
 \]
 \[
 D^{1}= E_{1}^{1,0}\oplus E_{1}^{0,1}
 \]
 \[
 D^2 = E_{1}^{1,1}
 \]
then
\begin{equation}
    \mathbb{H}^{1}(X,\mathcal{K}^{\bullet}) = \frac{\operatorname{Ker(D^1\longrightarrow D^2)}}{\operatorname{Im(D^0\longrightarrow D^1)}} = \frac{\operatorname{Ker(E_{1}^{1,0}\oplus E_{1}^{0,1}\longrightarrow E_{1}^{1,1})}}{\operatorname{Im(E_{1}^{0,0}\longrightarrow E_{1}^{1,0}\oplus E_{1}^{0,1})}}
\end{equation}
Yet another easier way to calculate it (for the case of curves) is from the following short exact sequence 
\begin{equation}
    0 \longrightarrow \operatorname{Coker}(E_{1}^{0,0}\longrightarrow E_{1}^{1,0}) \longrightarrow \mathbb{H}^{1}(X,\mathcal{K}^{\bullet})\longrightarrow E_{1}^{0,1}\longrightarrow 0
\end{equation}
hence from the splitting of this sequence 
\begin{equation}
    \mathbb{H}^{1}(X,\mathcal{K}^\bullet) = \operatorname{Coker}(E_{1}^{0,0}\longrightarrow E_{1}^{1,0}) \oplus E_{1}^{0,1}
\end{equation}
To see how this interpretation (80) is equivalent to the previously mentioned (79), we write the totalized complex
\begin{equation}
    D^{0}= E^{0,0} \overset{d^{0}= (d_{1},0)}{\longrightarrow} D^{1}= E_{1}^{1,0}\oplus E_{1}^{0,1}\overset{d^{1}= (0,d_{1}')}{\longrightarrow}D^{2}=E_{1}^{1,1}
\end{equation}
where 
\[
d_{1}: E_{1}^{0,0}\longrightarrow E_{1}^{1,0}
\]
which will be the connection $\nabla = d + \omega \wedge$ in our case and
\[
d_{1}': E_{1}^{0,1}\longrightarrow E_{1}^{1,1}
\]
Then 
\begin{equation}
    \operatorname{Ker}[d^{1}= (0,d_{1}')] = E_{1}^{1,0} \oplus \operatorname{Ker}(d_{1}':E_{1}^{0,1}\longrightarrow E_{1}^{1,1})
\end{equation}
\begin{equation}
\operatorname{Im}[d^{0}= (d_{1},0)] = \operatorname{Im}(d_{1}: E_{1}^{0,0}\longrightarrow E_{1}^{1,0}) \oplus 0
\end{equation}
In general, if the kernel $\operatorname{Ker[d_{1}': E_{1}^{0,1} \longrightarrow E_{1}^{1,1}]}$ is non-trivial, it should be interpreted as a term in the next page of the spectral sequence that is, $\operatorname{Ker[d_{1}': E_{1}^{0,1} \longrightarrow E_{1}^{1,1}]} = E_{2}^{0,1}$.\footnote{Since \[E_{2}^{p,q} = \frac{\operatorname{Ker}[d: E_{1}^{p,q}\longrightarrow E_{1}^{p+1,q}]}{\operatorname{Im}[d:E_{1}^{p-1,q}\longrightarrow E_{1}^{p,q}]}\]}\\

The first hypercohomology is then the cohomology at $D^1$ of the totalized complex
\begin{equation}
    \mathbb{H}^{1}(X,\mathcal{K}^{\bullet}) = \operatorname{H}^{1}(X,D^{\bullet}) = \frac{\operatorname{Ker}(d^{1})}{\operatorname{Im}(d^{0})}
\end{equation}
\[
= \frac{E_{1}^{1,0}\oplus \operatorname{Ker}[d_{1}':E_{1}^{0,1}\longrightarrow E_{1}^{1,1}]}{\operatorname{Im}[d_{1}:E_{1}^{0,0}\longrightarrow E_{1}^{1,0}]}
\]
\[
\simeq \frac{E_{1}^{1,0}}{\operatorname{Im}[d_{1}:E_{1}^{0,0}\longrightarrow E_{1}^{1,0}]} \oplus \operatorname{Ker}[d_{1}':E_{1}^{0,1}\longrightarrow E_{1}^{1,1}]
\]
\[
=\operatorname{Coker}[d_{1}:E_{1}^{0,0}\longrightarrow E_{1}^{1,0}] \oplus E_{2}^{0,1}
\]
For the equivalence between (79) and (80), it must be assumed that the map $d_{1}':E_{1}^{0,1}\longrightarrow E_{1}^{1,1}$ is a zero map so that the kernel of this map is the entire $E_{1}^{0,1}$. Then we get from (88)
\begin{equation}
    \mathbb{H}^{1}(X,\mathcal{K}^{\bullet}) = \operatorname{Coker}[d_{1}:E_{1}^{0,0}\longrightarrow E_{1}^{1,0}] \oplus E_{1}^{0,1}
\end{equation}
from which we can write the short exact sequence (80).
\\

Let us get back to our example computation. For us, the relevant spectral sequence page is 
\begin{center}
\begin{tikzpicture}
  \draw (0,0) rectangle (6,6);
  \draw (3,0) -- (3,6);
  \draw (0,3) -- (6,3);

  \node at (1.5,4.5) {$\operatorname{H}^{0}(E,\mathcal{O}_{E}) $};
  \node at (4.5,4.5) {$\operatorname{H}^{1}(E,\mathcal{O}_{E})$};
  \node at (1.5,1.3) {$\operatorname{H}^{0}(E,\Omega_{E}^{1}(\operatorname{log}D))$};
  \node at (4.5,1.3) {$\operatorname{H}^{1}(E,\Omega_{E}^{1}(\operatorname{log}D)$};
\end{tikzpicture}
\end{center}
We see that the zeroth cohomology is 
\begin{equation}
      \mathbb{H}^{0}(X,\mathcal{K}^{\bullet}) = \operatorname{Ker}(\operatorname{H}^{0}(E,\mathcal{O}_E)\overset{\nabla}{\longrightarrow}\operatorname{H}^{0}(E,\Omega_{E}^{1}(\operatorname{log}D))) = \operatorname{H}_{log,\omega}^{0}(E)
\end{equation}
which is zero, as was argued above too, because of the fact that $\operatorname{H}^{0}(E,\mathcal{O}_E)\simeq \mathbb{C}$.\\
The assumption that the map $d_{1}': E_{1}^{0,1}\longrightarrow E_{1}^{1,1}$ being a zero map can be made for our case using Deligne's logarithmic Hodge-to-de Rham degeneration of the spectral sequeuce. So the map $\operatorname{H}^{1}(E,\mathcal{O}_{E})\longrightarrow \operatorname{H}^{1}(E,\Omega_{E}^{1}(\operatorname{log}D))$ is a zero map and hence we can use (89) to compute 
\begin{equation}
    \mathbb{\operatorname{H}}^{1}_{log,\omega}(E) = \operatorname{coker}(\operatorname{H}^{0}(E,\mathcal{O}_{E})\longrightarrow \operatorname{H}^{0}(E,\Omega_{E}^{1}(\operatorname{log}D))) \oplus \operatorname{H}^{1}(E,\mathcal{O}_{E})
\end{equation} 
Now we may show that for two isomorphic algebraic twisting parameters, the algebraic twisted cohomologies are isomorphic. The idea is to "\textit{algebraize}" the \textbf{proof 2} of theorem 1. Just like  the first proof of that theorem, it is easier to proof this for the algebraic case locally. Since our algebraic varieties are projective, we may first look at their affine patches. Take any affine (open in the Zariski topology) $U = \operatorname{Spec} A \subset{X}$, we can take $g\in A^{\times}$. Then the algebraic derivative acting as 
\[
d(g\cdot g^{-1}) = d(g)\cdot g^{-1} + g\cdot d(g^{-1}) = 0 
\] 
\[ \Rightarrow d(g^{-1}) = -g^{-2}d(g) \] 
We recall that logarithmic structure on a variety $X$ is characterized by defining sheafification of the presheaf of monoids, $\mathcal{M}_X$ on the \'{e}tale or Zariski site of $X$. In addition, there is a sheaf homomorphism 
\begin{equation}
    \alpha: \mathcal{M}_{X}\longrightarrow \mathcal{O}_{X}
\end{equation}
where we think of the structure sheaf $\mathcal{O}_{X}$ as a sheaf of multiplicative monoids, such that this map reduces to a sheaf isomorphism $\alpha: \mathcal{M}_{X}^{\times}\longrightarrow \mathcal{O}_{X}^{\times} $ under the reduction $\mathcal{M}_{X}^{\times}\subset{\mathcal{M}_{X}},\mathcal{O}_{X}^{\times}\subset \mathcal{O}_{X}$. 
We have the following algebraic analog 
\begin{theorem}
    On the logarithmic (be it abelian or not) variety $(X,\mathcal{M}_{X})$, for two cohomologous logarithmic algebraic 1-form 
    \[
    \psi_{2} = \psi_{1} + \frac{dg}{g}
    \]
    for any $g\in\mathcal{O}_{X}^{\times}$, the twisted cohomologies are (locally) isomorphic. 
\end{theorem}
\textit{Proof :} Define the algebraic map 
\[
T : \eta \mapsto g^{-1}\cdot\eta
\]
Then we may compute 
\begin{align*}
    d_{\psi_2} (T(\eta)) &= (d + \psi_2\wedge)(g^{-1}\cdot\eta)\\
    &= d(g^{-1}\cdot\eta) + \psi_2 \wedge (g^{-1}\cdot\eta)\\
    &= -g^{-1}\frac{dg}{g}\wedge\eta + g^{-1}d\eta + \psi_{2}\wedge g^{-1}\cdot\eta\\
    &= g^{-1}\cdot(d\eta + [\psi_2 - \frac{dg}{g}]\wedge\eta)\\
    &= g^{-1}\cdot(d\eta + \psi_{1}\wedge\eta)
\end{align*} 
And thus \[  d_{\psi_2} (T(\eta)) = T (d_{\psi_1}\eta)
\]
and $T$ is an algebraic chain map. Then for locally defined $g\in \mathcal{O}_{X}^{\times}$, the algebraic twisted cohomologies are isomorphic.\textbf{QED}.\\
\section{Concluding remarks} We performed an elementary study and perhaps a different (than \cite{kita1994vanishing}) way of viewing twisted algebraic cohomologies on algebraic varieties. Our study was inspired by the perspectives presented in \cite{meng2018explicit,meng2020morse,meng2023morse}. We have not included many applications of such cohomologies. Instead, we have tried to motivate its definition and have included easy computations. We have made comparisons between analytic and algebraic twisting and the corresponding twisting parameters.
\bibliographystyle{abbrv}
\bibliography{refer}
\end{document}